\documentclass[11pt]{article}

\usepackage{amssymb}
\usepackage{amsmath} 
\usepackage{amsthm}
\allowdisplaybreaks[2]
 
\newcommand{\Ps}{\mathcal P}
\newcommand{\A}{\mathcal A}
\newcommand{\B}{\mathcal B}
\newcommand{\C}{\mathcal C}
\newcommand{\D}{\mathcal D}
\newcommand{\Os}{\mathcal O}

\theoremstyle{plain}
   \newtheorem {thm}{Theorem}
   \newtheorem{claim}[thm]{Claim}

\title{A Combinatorial Proof of a\\
Partition Identity of Andrews and Stanley}

\author{Andrew V. Sills\\
Department of Mathematics\\
Hill Center--Busch Campus\\ 
Rutgers University\\
Piscataway, NJ 08854-8019\\
\texttt{http://www.math.rutgers.edu/\~{}asills}\\
\texttt{asills@math.rutgers.edu}\\
2000 AMS subject code: 05A17
}
\date{Submitted January 26, 2004\\
      Revised February 5, 2004\\
      Accepted March 4, 2004}

\begin{document}
\maketitle
\begin{abstract}
In his paper, ``On a Partition Function of Richard Stanley," George Andrews proves a certain
partition identity analytically and asks for a combinatorial proof.  This paper provides 
the requested combinatorial proof.
\end{abstract}

\section{Introduction}
In the October 2002 issue of \textit{American Mathematical Monthly},
Richard Stanley posed a problem on partitions~\cite{rs}. 
In~\cite{gea:stanley}, George Andrews studies
the function $S(r,s;n) =$ the number of partitions $\pi$ of $n$ such that 
$\pi$ has $r$ odd parts and the conjugate $\pi'$ of $\pi$ has $s$ odd parts,
and proves a result from which he solves Stanley's
problem as a corollary.

 Andrews also states some additional interesting corollaries, including

\begin{thm}\label{mainresult}
\begin{equation}\label{ASI}
 \sum_{n,r,s\geqq 0} S(r,s;n) z^r y^s q^n
= \prod_{j=1}^\infty \frac{(1+yzq^{2j-1})}{(1-q^{4j})(1-z^2 q^{4j-2})(1-y^2 q^{4j-2})} ,
\end{equation}
\end{thm} 
\noindent which he proves analytically as the limiting case of a certain polynomial 
identity~\cite[Theorem 1]{gea:stanley}.  At the end of
the paper, Andrews states that (\ref{ASI}) ``cries out for a combinatorial proof."
The purpose of this paper is to present such a proof.  

\section{Definitions and Notation}
The following defintions and notations follow Macdonald~\cite{igm}.

 A \emph{partition} $\pi$ of an integer $n$ 
is a nonincreasing sequence of nonnegative integers
containing only finitely many nonzero terms such that the sum of
the terms is $n$.  Thus,
  \begin{equation} \pi = (\pi_1, \pi_2, \pi_3, \dots ), \label{defptn}
  \end{equation} with
  \[ \pi_1 \geqq \pi_2 \geqq \pi_3 \geqq \dots\geqq 0 \] 
  \[ \sum_{i=1}^\infty \pi_i = n, \] and
$\pi_i \in \mathbb Z.$
Since the tail of such a sequence must, by definition, be an infinite
string of zeros, it will be convenient to suppress this in the notation.
Thus the partition $(2,2,1,1,1,0,0,0,0,0,\dots)$ is normally written
$(2,2,1,1,1)$.  In any event, no distinction shall be drawn among 
sequences written with or without any number of trailing zeros. 

The nonzero terms $\pi_i$ in (\ref{defptn}) are called the \emph{parts} of the
partition $\pi$.  
The number of parts of $\pi$ is called the \emph{length} of $\pi$ and is
denoted $l(\pi)$.  Additionally, I will follow Stanley~\cite{rs} and
Andrews~\cite{gea:stanley} and let $\Os(\pi)$ denote the number of 
\emph{odd} parts in the partition $\pi$. 

 The set of all partitions is denoted $\mathcal{P}$.  
 
 An alternate notation for a partition $\pi$ 
is \[ \pi = \langle 1^{m_1} 2^{m_2} \cdots r^{m_r} \dots \rangle \]
which means that exactly $m_i$ of the parts of $\pi$ are equal to $i$.
The number $m_i = m_i(\pi)$ is
called the \emph{multiplicity} of $i$ in $\pi$.  
Thus, 
\[ (6,5,5,3,3,3,3,2) = \langle 2\ 3^4\ 5^2\ 6 \rangle. \]

 If $\pi = (\pi_1, \pi_2, \pi_3, \dots )$ and $\lambda = (\lambda_1, \lambda_2,
\lambda_3, \dots )$ are two partitions, their \emph{sum} is defined
termwise:
  \[ (\pi + \lambda)_i := \pi_i + \lambda_i \] 
for $i\in\mathbb{Z}_+$.

 The \emph{union} of two partitions $\lambda$ and $\pi$ is obtained
by merging the entries of $\lambda$ with those of $\pi$ and arranging
the resulting entries in nonincreasing order, for example,
   \[ (3,3,2,1) \cup (5,3,2,2) = (5,3,3,3,2,2,2,1) .\] 

 The \emph{conjugate} of a partition $\pi=(\pi_1,\pi_2,\pi_3,\dots)$, 
denoted $\pi'=(\pi'_1, \pi'_2, \pi'_3, \dots )$ is partition such that
   \begin{equation}\label{ConjDef}
     \pi'_i = \sum_{j\geqq i} m_j (\pi).
   \end{equation}
Note that $\pi'' = \pi$, $\pi_1 = l(\pi')$, and
\begin{equation}\label{MultConj}
   m_{i} (\pi') = \pi_{i} - \pi_{i+1}.
\end{equation}
Usually, the conjugate of a partition is thought of in terms
of the transposition of its \emph{Ferrers graph} or \emph{Young diagram};
see e.g.~\cite[p. 6]{Andrews1} or \cite[p. 2]{igm}.   

 Finally, define the parity function
 \[ P(i) = \left\{ \begin{array}{ll}
                        0 &\mbox{if $i$ is even} \\
                        1 &\mbox{if $i$ is odd.} 
                     \end{array} \right. \]
\section{Combinatorial Proof of Theorem~\ref{mainresult}}
\subsection{The Plan}
I will start with the set of all partitions $\mathcal P$ counted in such
a way so that it is easily seen to be generated by the infinite product
on the right hand side of (\ref{ASI}).  I will then map 
$\Ps$ bijectively onto itself while keeping track of enough of the internal
counting to see that (\ref{ASI}) holds.

  In order to keep the presentation as straightforward as possible, the
mapping will be presented in three simple steps (the maps $\alpha$, $\beta$ and $\gamma$), and the bijection will
be the composition of these three maps.

\subsection{The Generating Function}
 By standard elementary reasoning (see~\cite[Chapters 1--2]{Andrews1}), 
\[ \prod_{j=1}^\infty \frac{1}{1-q^{4j}}\] 
is the generating function for
partitions into parts $\equiv 0\pmod{4}$, 
\[ \prod_{j=1}^\infty \frac{1}{1-z^2 q^{4j-2}}\] 
is the generating function
for partitions into parts $\equiv 2\pmod{4}$, where each part is
counted with a weight of $z^2$, and 
\[ \prod_{j=1}^\infty
 \frac{1+yxq^{2j-1}}{1-y^2 q^{4j-2}}\]
is the generating
function for partitions into odd parts where each part is counted with
weight $y$ and each distinct integer of odd multiplicity is counted with
weight $x$.

 Thus, \[ \sum_{n\geqq 0} R(t,\rho,s;n) x^t z^{2\rho} y^s q^n  =
 \prod_{j=1}^\infty \frac{(1+yxq^{2j-1})}{(1-q^{4j})(1-z^2q^{4j-2})
  (1-y^2 q^{4j-2})}, \] 
where $R(n) =$ the number of partitions of $n$ with 
$s$ odd parts, $\rho$ parts congruent to
$2\pmod{4}$ and $t$ different odd integers of odd multiplicity.                       

  Upon setting $x$ equal to $z$ and renaming $2\rho + t$ as $r$, we
see that the infinite product of (\ref{mainresult}) generates
partitions with exactly $s$ odd parts and such that
twice the number of parts congruent to $2\pmod{4}$ plus 
the number of different odd integers of odd multiplicity is 
exactly $r$.  Thus, by mapping each of these partitions to a 
partition with exactly $r$ odd parts and $s$ odd parts in its conjugate,
we will have a bijective proof that (\ref{ASI}) holds. 

\subsection{The Mapping}
  Start with an arbitrary partition $\kappa = 
\langle 1^{m_1(\kappa)} 2^{m_2(\kappa)} 3^{m_3(\kappa)} \cdots \rangle$.
Define a map 
  \[ \alpha : \Ps \rightarrow \A \times \B, \]
where $\A$ is the set of all partitions with no repeated odd parts
and $\B =\Ps \setminus \A$ = the set of partitions with only odd parts of
even multiplicity, by
 \[ \alpha(\kappa) = (\lambda, \mu), \]
where 
\begin{equation}\label{lambdadef} \lambda = \langle 1^{P(m_1(\kappa))} 2^{m_2(\kappa)} 
  3^{P(m_3(\kappa))} 4^{m_4(\kappa)} \cdots \rangle, 
\end{equation}
and \[ \mu = \langle 1^{m_1(\kappa)-P(m_1(\kappa))} 
3^{m_3(\kappa) - P(m_3(\kappa))} 5^{m_5(\kappa) - P(m_5(\kappa))} \cdots 
\rangle. \] 
Informally, $\alpha$ separates $\kappa$ into two partitions $\lambda$
and $\mu$, where $\lambda$ gets all of the even parts of $\kappa$ and
one copy of each odd part of odd multiplicity in $\kappa$, and $\mu$ gets
what is left over, i.e. even quantities of odd parts. 

  Notice that 
\begin{equation} \Os(\lambda) = t \end{equation} and
\begin{equation} 
    l(\mu) = \Os(\mu) =  s-t. 
\label{OddMu}\end{equation}

  Define the map $\beta_1: \A\rightarrow \C$
so that $\beta_1(\lambda) := \zeta = (\zeta_1, \zeta_2, \dots)$ where, for $i\in\mathbb Z_+$,
\[ \zeta_{2i-1}= \lfloor \frac{\lambda_i +1}{2} \rfloor \] and
\[ \zeta_{2i}  = \lfloor \frac{\lambda_i   }{2} \rfloor \]
or equivalently,
\[ m_{i}(\zeta) = P(m_{2i-1}(\kappa)) + 2 m_{2i}(\kappa) + P(m_{2i+1}(\kappa)) .\]

  Informally, this means that every even part $2j$ of $\lambda$ is mapped by $\beta_1$ to a pair of $j$'s and each and every odd part $2j+1$ is mapped to the pair $j+1$, $j$.  So, $\C$ is the set of partitions 
$\zeta= (\zeta_1, \zeta_2, \dots)$ such that for all $i\in\mathbb Z_+$,
 \begin{equation}\label{ZetaChar}
   \zeta_{2i-1} - \zeta_{2i} \leqq 1. 
 \end{equation}
 
 Define the map $\beta_2: \B\rightarrow \D$ so that
$\xi:= \beta_{2} (\mu) = \mu'$.

Note that 
$\D$ is the set of partitions $\xi$ such that for all $i\in\mathbb Z_{+}$,
$\xi_{2i} = \xi_{2i+1},$ and $\xi_{i}$ is even.

 Also, $\Os(\zeta) = 2\rho + t$, $\Os(\xi) = 0$,
and $\Os(\xi') = \Os(\mu) = s - t$.
\begin{claim}\label{claim1} 
$\Os(\zeta') = t$
\end{claim}
\begin{proof}[Proof of Claim~\ref{claim1}] To see this, notice that by
(\ref{ZetaChar}), the only way that $\zeta'_{1}$ can be odd is if
1 is a part of $\lambda$ (which of course means that $m_1(\kappa)$ is
odd).  This is so because every part of $\lambda$ \emph{other than a 1} is
mapped to a pair of parts, but $1$ is mapped to $(1,0)$, which, by
definition is only \emph{one part} since $0$ does not count as a part.

  Next, the only way that $\zeta'_2$ can be odd is if $\lambda$ contains
a $3$; and in general, the only way that $\zeta'_{i}$ can be odd is
if $\lambda$ contains $2i-1$ as a part (which means that 
$m_{2i-1}(\kappa)$ is odd).  
\end{proof}

 Finally, define a map $\gamma: \C \times \D \rightarrow \Ps$,
by $\pi = \gamma(\zeta,\xi):= \zeta + \pi$.   Thus, we have
$\Os(\pi) = \Os(\zeta) + \Os(\xi) = 2\rho + t =: r$ 

\begin{claim}\label{claim2}
$\Os(\pi') = \Os(\zeta') + \Os(\xi') = s$ as desired, i.e. no
odd parts in the conjugate are ``lost" through the addition
of the partitions $\zeta$ and $\xi$. 
\end{claim}
\begin{proof}[Proof of Claim~\ref{claim2}]
\begin{eqnarray*}
\Os(\pi')&=& \sum_{i\geqq 1} m_{2i-1}(\pi')\\
    &=& \sum_{i\geqq 1} (\pi_{2i-1} - \pi_{2i}) \\
    &=& \sum_{i\geqq 1} \big( (\zeta_{2i-1} + \xi_{2i-1}) - (\zeta_{2i}+\xi_{2i}) \big)\\
    &=& \sum_{i\geqq 1} \left(
     \left(\lfloor \frac{\lambda_i + 1}{2} \rfloor + \mu'_{2i-1}\right)
      -\left(\lfloor \frac{\lambda_i    }{2} \rfloor + \mu'_{2i}\right)
       \right)\\
    &=& \sum_{i\geqq 1}\left(\lfloor \frac{\lambda_i + 1}{2} \rfloor 
      - \lfloor \frac{\lambda_i    }{2} \rfloor 
      + \mu'_{2i-1} - \mu'_{2i} \right)\\ 
   &=& \sum_{i\geqq 1} \lfloor \frac{\lambda_i + 1}{2} \rfloor 
      - \lfloor \frac{\lambda_i    }{2} \rfloor 
      + m_{2i-1}(\mu) \quad\mbox{by (\ref{MultConj})}\\
   &=& \Os(\lambda) + \Os(\mu) \\
   &=& t + (s-t)\\
   &=& s,  
\end{eqnarray*}
where the third to last equality follows from the fact that
$\lfloor \frac{\lambda_i + 1}{2} \rfloor 
      - \lfloor \frac{\lambda_i    }{2} \rfloor$ equals $0$ if $i$ is
  even, and equals $1$ if $i$ is odd.
\end{proof}

 Now let us consider the invertibility of each of the maps.  
It is easy to see that
$\alpha$ is invertible; $\alpha^{-1}(\lambda,\mu) = \lambda \cup \mu$.
$\beta_2^{-1} = \beta_2$ since conjugation
is an involution. 
$\beta_1^{-1}(\zeta) = \lambda = (\lambda_{1}, \lambda_{2}, \dots)$, where
 \[ \lambda_{i} = \zeta_{2i-1} + \zeta_{2i} \]
for $i\in\mathbb Z_{+}$.

  The invertibility of $\gamma$, however, is more subtle.  Given any
partition $\pi$, we need to ``split it" into the sum of two partitions
$\zeta\in\C$ and $\xi\in\D$.  It is sufficient to recover
$\xi_1$ and  $\xi_{2i}$ for
$i\geqq 1$ since $\xi_{2i} = \xi_{2i+1}$ for all $i\geqq 1$, and once we
have $\xi$, we know $\zeta$ since $\zeta + \xi$ = $\pi$.
Of course, we need to do this given only the partition $\pi$.

  The easiest case is finding $\xi_1$ since we know that $\xi_1 = s - t$.
Both $s$ and $t$ are readily available to us in  $\pi$: Besides counting 
the number of odd parts in $\kappa$, $s$ counts
the number of odd parts in $\pi'$.  Now $t$ counts the number of odd
parts of odd multiplicity in $\kappa$ which translates into the number
of odd parts in $\lambda$, which translates into the number of pairs
$(\zeta_{2i-1}, \zeta_{2i})$ in $\zeta$ whose sum is odd, which in turn
translates to the number of pairs
$(\pi_{2i-1}, \pi_{2i})$ in $\pi$ whose sum is odd (since all parts of
$\xi$ are even).  Thus we can recover $\xi_1$ from $\pi$.

Now that we have both $\pi$ and $\xi_1$ in hand, we can recover the
rest of $\xi$:   
\begin{claim}\label{claim3}
For $i\geqq 1$, \[ \xi_{2i} = 
\xi_1 - \sum_{j=1}^i (\pi_{2i-1} - \pi_{2i})  + 
  \sum_{j=1}^i P(\pi_{2i-1} + \pi_{2i}).\]
\end{claim}
\begin{proof}[Proof of Claim~\ref{claim3}]
\begin{eqnarray*}
&& \xi_1 - \sum_{j=1}^i (\pi_{2j-1} - \pi_{2j}) 
  +\sum_{j=1}^i P(\pi_{2j-1} + \pi_{2j}) \\
&=& \xi_1 - \sum_{j=1}^i (\zeta_{2j-1} + \xi_{2j-1} - \zeta_{2j} - \xi_{2j})  +\sum_{j=1}^i P(\lambda_j)\\ 
&=& \xi_1- \sum_{j=1}^i \left( 
   \lfloor \frac{\lambda_j + 1}{2} \rfloor + \xi_{2j-1}
  -\lfloor \frac{\lambda_j    }{2} \rfloor - \xi_{2j} \right)
  +\sum_{j=1}^i P(\lambda_j)\\
&=& \xi_1 +\sum_{j=1}^i \left( P(\lambda_j) - \left(
    \lfloor \frac{\lambda_j + 1}{2} \rfloor
   -\lfloor \frac{\lambda_j    }{2} \rfloor \right) \right)
   - \sum_{j=1}^i (\xi_{2j-1} - \xi_{2j})\\
&=& \xi_1 + \sum_{j=1}^i 0 - (\xi_1 + \xi_2) - \cdots 
-(\xi_{2i-1}-\xi_{2i})\\
&=& \xi_1 - \xi_1 + (\xi_2 - \xi_3) + \dots +(\xi_{2i-2} - \xi_{2i-1}) + \xi_{2i}\\
&=& \xi_{2i}  
\end{eqnarray*}
\end{proof}

\section{Example}
Start with a partition 
\[ \kappa = \langle 1^5\ 2\ 3^3\ 4^2\ 5\ 6^2\ 7^2\ 9 \rangle. \]
Under $\alpha$, the parts of $\kappa$ are separated into $\lambda$
and $\mu$, where $\lambda$ gets all of the even parts of $\kappa$,
and one copy of each odd part of odd multiplicity and $\mu$ gets
what is left over, i.e. even quantities of odd parts:
\begin{gather*}
(\lambda,\mu) = (\langle 1\ 2\ 3\ 4^2\ 5\ 6^2\ 9 \rangle, 
          \langle 1^4\ 3^2\ 7^2 \rangle )\\
   =\big( (9,6,6,5,4,4,3,2,1), (7,7,3,3,1,1,1,1) \big)
\end{gather*}
  Next, $\beta_1$ takes each even part $2j$ of $\lambda$ to a pair
of $j$'s and each odd part $2j+1$ to the pair $j+1,j$, and $\beta_2$
maps $\zeta$ to its conjugate.
\begin{gather*}
(\zeta,\xi) = \big( (5,4,\ 3,3,\ 3,3,\ 3,2,\ 2,2,\  2,2,\ 2,1,\ 1,1,\ 1,0),
  (8,4,4,2,2,2,2) \big)
\end{gather*}
Finally, $\gamma$ merges $\zeta$ and $\xi$ into a single partition
$\pi$ via partition addition:
\[ \pi = (13,8,7,5,5,5,5,2,2,2,2,2,2,1,1,1,1) = 
\langle 1^4\ 2^6\ 5^4\ 7\ 8\ 13 \rangle\]

  Now let us consider the inverse map, pretending for the moment that
we only know $\pi$.  As in the general case, the hardest part will be
to separate $\pi$ into $\zeta$ and $\xi$.  
First of all, $\pi' = \langle 1^5\ 2\ 3^2\ 7^3\ 13\ 17
\rangle$, so $\Os(\pi') = s =12$.  Next, notice that 
$\pi = (13,8,\; 7,5,\; 5,5,\; 5,2,\; 2,2,\; 2,2,\; 2,1,\; 1,1,\; 1,0)$ has
$t=4$ pairs of opposite parity.  Thus we know that $\xi_1 = 12-4 = 8$.   

  Next, to find $\xi_2$, we take $\xi_1 -(\pi_1 - \pi_2) + P(\pi_1 +\pi_2) =
8-(13-8)+1 = 4$.  Thus $\xi_3 = 4$ also.   

  To find $\xi_4$, we take $\xi_1 - (\pi_1 - \pi_2 + \pi_3 - \pi_4) +
P(\pi_1 + \pi_2) + P(\pi_3 + \pi_4) = 8 - (13-8 + 7-5) + 1 + 0 =2.$
Thus $\xi_5= 2$, also.

  Similiarly, we find $\xi_6 = \xi_7 = 2$ and $\xi_i = 0$ for $i\geqq 8$.
Thus we have $\xi = (8,4,4,2,2,2,2)$ and from this we immediately get  
$\zeta = \langle 1^4\ 2^6\ 3^5 \ 4\ 5 \rangle$ by subtracting $\pi_i - \xi_i$
for $i\geqq 1.$

 Having gotten this far, the rest is simple.  $\lambda =
(5+4, 3+3, 3+3, 3+2, 2+2, 2+2, 2+1, 1+1, 1+0) = (9,6,6,5,4,4,3,2,1)$, and
$\mu = \xi' = (7,7,3,3,1,1,1,1).$  Finally, $\kappa = \lambda \cup \mu
 = (9,7,7,6,6,5,4,4,3,3,3,2, 1,1,1,1,1)$.

  Note: I have written a Maple package which performs all of the
mappings described in this paper, and also provides additional
examples.  Please visit 
\texttt{http://www.math.rutgers.edu/\~{}asills/papers.html} to
download it.

\section{Original Presentation}
This paper provides the full account of the talk I presented in the 
special session on $q$-series at the AMS sectional meeting in Baton
Rouge on March 15, 2003~\cite{avs}.
 
\section{Subsequent Developments}
It has come to my attention that Cilanne Boulet~\cite{cb} and 
Ae Ja Yee~\cite{ajy} 
have independently discovered additional 
combinatorial proofs of Theorem~\ref{mainresult}.
 
\section*{Acknowledgements}
I am grateful to George Andrews for initially supplying the problem 
and to the referees for their many helpful suggestions.  
I also wish to express my thanks to the organizers of the Baton Rouge
$q$-series special session,
Mourad E. H. Ismail and Stephen C. Milne, for the opportunity to
present this material.

\end{document}